\documentclass[10pt]{amsart}

\usepackage{enumerate}
\usepackage{graphicx}
\usepackage{float}
\usepackage{placeins}
\usepackage{mdframed}
\usepackage{amssymb}
\usepackage{esint}
\usepackage{cool}
\usepackage[all,cmtip]{xy}
\usepackage{mathtools}
\usepackage{amstext} 
\usepackage{array}   
\usepackage[shortlabels]{enumitem}

\newcolumntype{L}{>{$}l<{$}} 
\newcolumntype{C}{>{$}c<{$}}
\newtheorem{theorem}{Theorem}[section]
\newtheorem{lemma}[theorem]{Lemma}
\newtheorem{cor}[theorem]{Corollary}
\newtheorem{prop}[theorem]{Proposition}
\newtheorem{setup}[theorem]{Setup}
\theoremstyle{definition}
\newtheorem{definition}[theorem]{Definition}
\newtheorem{example}[theorem]{Example}

\newtheorem{notation}[theorem]{Notation}
\theoremstyle{remark}
\newtheorem{remark}[theorem]{Remark}

\newtheorem{the context}[theorem]{The Context}

\numberwithin{equation}{theorem}
\numberwithin{equation}{section}

\newcommand{\cat}[1]{\mathcal{#1}}

\newcommand{\rank}{\operatorname{rank}}

\newcommand{\soc}{\operatorname{Soc}}

\newcommand{\tor}{\operatorname{Tor}}
\newcommand{\im}{\operatorname{Im}}

\newcommand{\Ker}{\operatorname{Ker}}

\newcommand{\ideal}[1]{\mathfrak{#1}}
\newcommand{\m}{\ideal{m}}

\renewcommand{\geq}{\geqslant}
\renewcommand{\leq}{\leqslant}
\renewcommand{\ker}{\Ker}
\renewcommand{\hom}{\Hom}

\newcommand{\Hom}{\operatorname{Hom}}

\newcommand{\maps}[5]{\xymatrix{#1 \ar[r]^-{#3} & #2 \\
#4 \ar@{|->}[r] & #5 \\}}

\newcommand{\mfa}{\mathfrak{a}}

\def\w{\wedge}

\def\im{\operatorname{Im}}

\newcommand{\rk}{\textrm{rk}}

\begin{document}
\title[Trimming Complexes]{Trimming Complexes and Applications to Resolutions of Determinantal Facet Ideals}

\author{Keller VandeBogert }
\date{\today}

\maketitle

\begin{abstract}
    We produce a family of complexes called trimming complexes and explore applications. We demonstrate how trimming complexes can be used to deduce the Betti table for the minimal free resolution of the ideal generated by certain subsets of a generating set for an arbitrary ideal $I$. In particular, we compute the Betti table of the ideal obtained by removing an arbitrary generator from the ideal of submaximal pfaffians of a generic skew symmetric matrix $M$. We also explicitly compute the Betti table for the ideal generated by certain subsets of the generating set of the ideal of maximal minors of a generic $n \times m$ matrix. Such ideals are a subset of a class of ideals called determinantal facet ideals, whose higher degree Betti numbers had not previously been computed.
\end{abstract}

\section{Introduction}

Let $(R, \m , k)$ be a regular local ring with maximal ideal $\m$ and residue field $k$. Given an $\m$-primary ideal $I = (\phi_1 , \dots , \phi_n) \subseteq R$, one can ``trim" the ideal $I$ by, for instance, forming the ideal $(\phi_1 , \dots , \phi_{n-1} ) + \m \phi_n$. This process is used by Christensen, Veliche, and Weyman (see \cite{christensen2019trimming}) in the case that $R/I$ is a Gorenstein ring to produce ideals defining rings with certain Tor algebra classification, negatively answering a question of Avramov in \cite{avramov2012cohomological}. 

This trimming procedure also arises in classifying certain type $2$ ideals defining compressed rings. More precisely, it is shown in \cite{vandebogert2019structure} that every homogeneous grade $3$ ideal $I \subseteq k[x,y,z]$ defining a compressed ring with socle $\soc (R/I) = k(-s) \oplus k(-2s+1)$ is obtained by trimming a Gorenstein ideal. A complex is produced that resolves all such ideals; it is generically minimal. This resolution is then used to bound the minimal number of generators and, consequently, parameters arising in the Tor algebra classification.

In this paper, we generalize the resolution of \cite[Theorem 5.4]{vandebogert2019structure}, mentioned in the previous paragraph. To be precise, given an ideal $I = (\phi_1 , \dotsc , \phi_n)$ we construct a resolution of $(\phi_1 , \dotsc , \phi_{n-1}) + \mfa \phi_n$ for an arbitrary ideal $\mfa$; the setting explored in \cite{vandebogert2019structure} is for $I \subseteq k[x,y,z]$ a grade $3$ Gorenstein ideal and $\mfa = (x,y,z)$. Moreover, we wish to explore applications of these complexes in situations that are unrelated to the setting of the previous paragraph. One such case is for computing the graded Betti numbers of the ideal generated by certain subsets of the standard generating set of the ideal of maximal minors of a generic $n \times m$ matrix $M$ (see Theorem \ref{subres}). 

The homogeneous minimal free resolution of ideals generated by all minors of a given size of some matrix is well understood (see, for instance, \cite{bruns2006determinantal}). It is less well understood what the minimal resolution/Betti table of the ideal generated by \emph{subsets} of these minors must be. Certain classes of subsets have applications in algebraic statistics, including the adjacent $2$-minors of an arbitrary matrix and arbitrary subsets of a $2 \times n$ matrix are considered (see \cite{herzog2010binomial}, \cite{hocsten2004ideals} for the former case). The latter case has been studied by Herzog et al (see \cite{herzog2012ideals}); in particular, such ideals are always radical, and the primary decomposition and Gr\"obner basis are known. 

In \cite{ene2013determinantal}, so-called determinantal facet ideals are studied. Every maximal minor has an associated simplex, and a collection of minors can then be indexed by the facets of a certain simplicial complex $\Delta$ on the vertex set $\{ 1 , \dots , n \}$, for some $n$. Properties of the determinantal facet ideal may be deduced from properties of $\Delta$. A study of the homological properties of these ideals is conducted in \cite{herzog2017linear}; in particular, the Betti numbers of the linear strand of the minimal free resolution of these ideals is computed in terms of the $f$-vector of the associated clique complex.

In this paper, we consider a subset of the cases addressed in \cite{herzog2017linear}; however, we compute Betti numbers \emph{explicitly} in all degrees, instead of just the linear strand, and our formulas do not depend on any combinatorial machinery. We also deduce that the ideals under consideration are never linearly presented and hence never have linear resolutions. 

The paper is organized as follows. Sections \ref{trimcx} and \ref{ittrimcx} introduce the main machinery of the paper: the trimming complex and iterated trimming complex. We prove that these complexes are resolutions that are not necessarily minimal. However, due to the simple nature of the differentials, one can deduce the ranks appearing in the minimal free resolution of the ideal of interest.

In Section \ref{ressubs}, we show how to use the complex of Section \ref{trimcx} to resolve ideals generated by certain subsets of a minimal generating set of an arbitrary ideal $I$. As applications, we compute the Betti tables of the ideals obtained by removing a single generator from the ideal of submaximal pfaffians (see Proposition \ref{prop:btabforpfaffs}) and from the ideal of maximal minors of a generic $n \times m$ matrix $M$ (see Theorem \ref{thm:btabforsinglemaxlminor}). In Section \ref{maxlminors} we use the iterated trimming complex of Section \ref{ittrimcx} to compute the Betti tables of ideals obtained by removing certain additional generators from the generating set of the ideal of maximal minors of a generic $n \times m$ matrix $M$. As an application, we are able to deduce pieces of the $f$-vector of the simplicial complex associated to certain classes of uniform clutters.

\section{Trimming Complexes}\label{trimcx}

In this section, we introduce the notion of trimming complexes and show that, in fact, these complexes are resolutions. We begin by defining the quotient rings we aim to resolve and setting up the notation that we will use throughout the section.

\begin{setup}\label{setup1}
Let $R=k[x_1,\dots, x_n]$ be a standard graded polynomial ring over a field $k$. Let $I \subseteq R$ be a homogeneous ideal and $(F_\bullet, d_\bullet)$ denote a homogeneous free resolution of $R/I$. 

Write $F_1 = F_1' \oplus Re_0$, where $e_0$ generates a free direct summand of $F_1$. Using the isomorphism
$$\hom_R (F_2 , F_1 ) = \hom_R (F_2,F_1') \oplus \hom_R (F_2 , Re_0)$$
write $d_2 = d_2' + d_0$, where $d_2' \in \hom_R (F_2,F_1')$, $d_0 \in \hom_R (F_2 , Re_0)$. Let $\mfa$ denote any homogeneous ideal with
$$d_0 (F_2) \subseteq \mfa e_0,$$
and $(G_\bullet , m_\bullet)$ be a homogeneous free resolution of $R/\mfa$. 

Use the notation $K' := \im (d_1|_{F_1'} : F_1' \to R)$, $K_0 := \im (d_1|_{Re_0} : Re_0 \to R)$, and let $J := K' + \mfa \cdot K_0$.
\end{setup}

Our goal is to construct a resolution of the quotient ring $R/J$ as in Setup \ref{setup1}. Observe that the length of $G_\bullet$ does not have to equal the length of $F_\bullet$. 

\begin{prop}\label{colon}
Adopt notation and hypotheses as in Setup \ref{setup1}. Then
$$(K' : K_0 ) \subseteq \mfa.$$
\end{prop}

\begin{proof}
Let $r \in R$ with $r K_0 \subseteq K'$. By definition there exists $e' \in F_1'$ such that
$$d_1 ( e' + re_0) = 0.$$
By exactness of $F_\bullet$, there exists $f \in F_2$ with $d_2 ( f) = e' + re_0$. Employing the decomposition $d_2 = d_2' + d_0$, we find
$$d_0 (f) - re_0 = e' - d_2' (f) \in F_1' \cap Re_0 = 0$$
whence $d_0 (f) = re_0$. By selection of $\mfa$, we conclude $r \in \mfa$.
\end{proof}

\begin{prop}\label{1stmap}
Adopt notation and hypotheses as in Setup \ref{setup1}. Then there exists a map $q_1 : F_2 \to G_1$ such that the following diagram commutes:
$$\xymatrix{& F_2 \ar[dl]_-{q_1} \ar[d]^{d'_0} \\
G_1 \ar[r]_{m_1} & \mfa, \\}$$
where $d_0' : F_2 \to R$ is the composition
$$\xymatrix{F_2 \ar[r]^{d_0} & Re_0 \ar[r] & R\\},$$
and where the second map sends $e_0 \mapsto 1$.
\end{prop}

\begin{proof}
This follows directly from the fact that $F_2$ is projective.
\end{proof}

\begin{prop}\label{theyexist}
Adopt notation and hypotheses as in Setup \ref{setup1}. Then there exist maps $q_k : F_{k+1} \to G_{k}$ for all $k \geq 2$ such that the following diagram commutes:
$$\xymatrix{F_{k+1} \ar[d]_{q_k} \ar[r]^-{d_{k+1}} & F_k \ar[d]^{q_{k-1}} \\
G_k \ar[r]_{m_k} & G_{k-1} \\}$$
\end{prop}

\begin{proof}
We build the $q_k$ inductively. For $k=2$, observe that
$$m_1 \circ q_1 \circ d_3 = d'_0 \circ d_3 = 0,$$
so there exists $q_2 : F_3 \to G_2$ making the desired diagram commute. For $k>2$, we assume that $q_{k-1}$ has already been constructed. Then
$$m_{k-1} \circ q_{k-1} \circ d_{k+1} = q_{k-2} \circ d_k \circ d_{k+1} = 0,$$
so the desired map $q_k$ exists.
\end{proof}

\begin{theorem}\label{theres}
Adopt notation and hypotheses as in Setup \ref{setup1}. Then the mapping cone of the morphism of complexes
\begin{equation}\label{comx}
\xymatrix{\cdots \ar[r]^{d_{k+1}} &  F_{k} \ar[d]^{q_{k-1}}\ar[r]^{d_{k}} & \cdots \ar[r]^{d_3} & F_2 \ar[rr]^{d_2'} \ar[d]^{q_1} && F_1' \ar[d]^{d_1} \\
\cdots \ar[r]^{m_k} &G_{k-1} \ar[r]^{m_{k-1}} & \cdots \ar[r]^{m_2} & G_1 \ar[rr]^-{-m_1(-)\cdot d_1(e_0)} && R \\}\end{equation}
is acyclic and is a free resolution of $R/J$.
\end{theorem}

\begin{proof}
We first verify that the maps given in the statement of Theorem \ref{theres} form a morphism of complexes. To this end, it suffices only to show that the first square commutes. Let $f \in F_2$; moving counterclockwise around the first square, we see
\begin{align*}
    f & \mapsto - m_1 (q_1 (f)) \cdot d_1 (e_0) \\
    &= -d_0' (f) \cdot d_1 (e_0) \\
    &= -d_1 (d_0'(f) e_0) \\
    &= -d_1 (d_0 (f)) =  d_1 (d_2' (f) ) .\\
\end{align*}
Thus we have a well defined morphism of complexes. Let $q_\bullet$ denote the collection of vertical maps in \ref{comx}, $F_\bullet'$ the top row of \ref{comx}, and $G'_\bullet$ the bottom row of \ref{comx}. There is a short exact sequence of complexes:
$$0 \to G'_\bullet \to \textrm{Cone} (q_\bullet) \to F'_\bullet [-1] \to 0,$$
which induces the standard long exact sequence in homology. Using this long exact sequence of homology, $\textrm{Cone} (q_\bullet)$ will be a resolution of $R/J$ if:
\begin{enumerate}
    \item The complex $F_\bullet'$ is a resolution of $K'/(K'\cap K_0)$.
    \item The complex $G_\bullet'$ is a resolution of $R/(\mfa K_0)$.
    \item The induced map on $0$th homology
    $$\frac{K'}{K' \cap K_0} \to \frac{R}{\mfa K_0}$$
    is an injection.
\end{enumerate}
To prove $(1)$, observe that the top row of \ref{comx} appears as the bottom row in the short exact sequence of complexes
\begingroup\allowdisplaybreaks\begin{equation}\label{diag1}\xymatrix{&& 0 \ar[d] & 0 \ar[d] & \\
& 0 \ar[r] & Re_0 \ar[r]^-{d_1} \ar[d] & K_0 \ar[d] \ar[r] & 0 \\
\cdots \ar[r] & F_2 \ar@{=}[d] \ar[r]^{d_2} & F_1 \ar[d] \ar[r]^{d_1} & I \ar[d] \ar[r] & 0 \\
\cdots \ar[r] & F_2 \ar[d]  \ar[r]^{d_2'} & F_1' \ar[d] \ar[r]^{d_1} & \frac{K'}{K' \cap K_0} \ar[d] \ar[r] & 0 \\
& 0 & 0 & 0 \\}\end{equation}\endgroup
The top row of \ref{diag1} is exact since $R$ is a domain. The middle row is exact since $F_\bullet$ is a resolution of $R/I$, so the bottom row must also be exact. Notice that the rightmost column is exact since $I /K_0 = (K'+K_0)/K_0 \cong K'/ (K' \cap K_0)$.

Similarly, $(2)$ holds because $R$ is a domain. More precisely, if $g\in G$ and $m_1( g) \cdot d_1 (e_0) = 0$, then $m_1 (g) = 0$. Since $G_\bullet$ is exact by assumption, $g \in \im(d_2)$.

Lastly, to prove $(3)$, simply observe that $K' \cap \mfa K_0 \subseteq K' \cap K_0$. 
\end{proof}

\begin{definition}\label{def:trimcx}
The \emph{trimming complex} associated to the data of Setup \ref{setup1} is the resolution of Theorem \ref{theres}.
\end{definition}

\begin{remark}
Notice that in Definition \ref{def:trimcx}, the associated trimming complex depends on a chosen generating set for $I$, not just the ideal itself.
\end{remark}

In general, the trimming complex associated to the data of Setup \ref{setup1} need not be minimal. However, the following Corollary allows us to deduce the (graded) Betti numbers even for a nonminimal resolution.

\begin{cor}\label{torrk}
Adopt notation and hypotheses of Setup \ref{setup1}. Assume furthermore that the resolutions $F_\bullet$ and $G_\bullet$ are minimal. Then for $i \geq 2$,
$$\dim_k \tor_i^R (R/J , k) = \rank F_i + \rank G_i - \rank (q_{i-1} \otimes k) - \rank (q_i \otimes k),$$
and
$$\mu (J) = \mu(I) + \mu(\mfa) - 1 -\rank (q_1 \otimes k).$$ 
\end{cor}

\begin{proof}
Resolve $R/J$ by the mapping cone of the diagram in Theorem \ref{theres}, and let $\ell_i$ denote the $i$th differential. Then for $i\geq 2$,
$$\dim_k \tor_i^R (R/J , k) = \dim_k \ker (\ell_i\otimes k )/\im (\ell_{i+1} \otimes k).$$
Since the resolutions $F_\bullet$ and $G_\bullet$ are minimal by assumption,
$$\rank (\im (\ell_{i+1}) = \rank (q_{i} \otimes k), \ \textrm{and} $$
$$\quad \rank (\ker (\ell_i \otimes k)) = \rank F_i + \rank G_i - \rank (q_{i-1} \otimes k).$$
For the latter claim, observe that $\ell_1 \otimes k = 0$, so
$$\dim_k \tor_1 (R/J , k ) = \rank F_1' + \rank G_1 - \rank (q_1 \otimes k).$$
Since $\rank F_1' = \mu (I) - 1$ and $\rank G_1 = \mu (\mfa)$, the claim follows after recalling $\dim_k \tor_1 (R/J , k ) = \mu (J)$.
\end{proof}

\begin{remark}\label{rk:gradedBnos}
Observe that in the setting of Corollary \ref{torrk}, if the resolutions $F_\bullet$ and $G_\bullet$ are also graded, then we may restrict the equalities to homogeneous pieces to find the graded Betti numbers as well.
\end{remark}

\section{Iterated Trimming Complexes}\label{ittrimcx}

In this section, we consider an iterated version of the data of Setup \ref{setup1}, and construct a similar resolution. We conclude this section with a concrete example illustrating the construction.

\begin{setup}\label{setup4}
Let $R=k[x_1,\dots, x_n]$ be a standard graded polynomial ring over a field $k$. Let $I \subseteq R$ be a homogeneous ideal and $(F_\bullet, d_\bullet)$ denote a homogeneous free resolution of $R/I$. 

Write $F_1 = F_1' \oplus \Big( \bigoplus_{i=1}^t Re_0^i \Big)$, where, for each $i=1, \dotsc , t$, $e^i_0$ generates a free direct summand of $F_1$. Using the isomorphism
$$\hom_R (F_2 , F_1 ) = \hom_R (F_2,F_1') \oplus \Big( \bigoplus_{i=1}^t \hom_R (F_2 , Re^i_0) \Big)$$
write $d_2 = d_2' + d_0^1 + \cdots + d^t_0$, where $d_2' \in \hom_R (F_2,F_1')$ and $d^i_0 \in \hom_R (F_2 , Re^i_0)$. 

For each $i=1, \dotsc , t$, let $\mfa_i$ denote any homogeneous ideal with
$$d^i_0 (F_2) \subseteq \mfa_i e^i_0,$$
and $(G^i_\bullet , m^i_\bullet)$ be a homogeneous free resolution of $R/\mfa_i$. 

Use the notation $K' := \im (d_1|_{F_1'} : F_1' \to R)$, $K^i_0 := \im (d_1|_{Re^i_0} : Re^i_0 \to R)$, and let $J := K' + \mfa_1 \cdot K^1_0+ \cdots + \mfa_t \cdot K_0^t$.
\end{setup}

The next few Propositions are directly analogous to those of the previous section; the proofs are omitted since they are identical.

\begin{prop}\label{it1stmap}
Adopt notation and hypotheses of Setup \ref{setup4}. Then for each $i=1,\dotsc, t$ there exist maps $q^i_1 : F_2 \to G^i_1$ such that the following diagram commutes:
$$\xymatrix{& F_2 \ar[dl]_-{q^i_1} \ar[d]^{{d^i_0}'} \\
G_1 \ar[r]_{m^i_1} & \mfa_i, \\}$$
where ${d^i_0}' : F_2 \to R$ is the composition
$$\xymatrix{F_2 \ar[r]^{d^i_0} & Re^i_0 \ar[r] & R\\},$$
and where the second map sends $e^i_0 \mapsto 1$. 
\end{prop}

\begin{prop}\label{ittheyexist}
Adopt notation and hypotheses as in Setup \ref{setup4}. Then for each $i=1, \dots , t$ there exist maps $q^i_k : F_{k+1} \to G^i_{k}$ for all $k \geq 2$ such that the following diagram commutes:
$$\xymatrix{F_{k+1} \ar[d]_{q^i_k} \ar[r]^-{d_{k+1}} & F_k \ar[d]^{q^i_{k-1}} \\
G^i_k \ar[r]_{m^i_k} & G^i_{k-1} \\}$$
\end{prop}

\begin{theorem}\label{itres}
Adopt notation and hypotheses as in Setup \ref{setup4}. Then the mapping cone of the morphism of complexes
\begin{equation}\label{itcomx}
\xymatrix{\cdots \ar[r]^{d_{k+1}} &  F_{k} \ar[ddd]^{\begin{pmatrix} q_{k-1}^1 \\
\vdots \\
q_{k-1}^t \\
\end{pmatrix}}\ar[r]^{d_{k}} & \cdots \ar[r]^{d_3} & F_2 \ar[rrrr]^{d_2'} \ar[ddd]^{\begin{pmatrix} q_1^1 \\
\vdots \\
q_1^t \\
\end{pmatrix}} &&&& F_1' \ar[ddd]^{d_1} \\
&&&&&&& \\
&&&&&&& \\
\cdots \ar[r]^-{\bigoplus m^i_k} & \bigoplus_{i=1}^t G^i_{k-1} \ar[r]^-{\bigoplus m^i_{k-1}} & \cdots \ar[r]^-{\bigoplus m^i_2} & \bigoplus_{i=1}^t G^i_1 \ar[rrrr]^-{-\sum_{i=1}^t m^i_1(-)\cdot d_1(e^i_0)} &&&& R \\}\end{equation}
is a free resolution of $R/J$.
\end{theorem}

The proof of Theorem \ref{itres} follows from iterating the construction of Theorem \ref{theres}; however, there is some careful bookkeeping needed to deduce that the mapping cone of \ref{itcomx} can be obtained by iterating the mapping cone construction of Theorem \ref{theres}. 



\begin{proof}[Proof of Theorem \ref{itres}]
Adopt notation and hypotheses of Setup \ref{setup4}. Let $(F_\bullet^1,d^1_\bullet)$ denote the complex of Theorem \ref{theres} applied to the direct summand $Re_0^1$ of $F_1$; that is, the mapping cone of:
\begin{equation}
\xymatrix{\cdots \ar[r]^{d_{k+1}^1} &  F_{k} \ar[d]^{q_{k-1}^1}\ar[r]^{d_{k}^1} & \cdots \ar[r]^{d_3^1} & F_2 \ar[rr]^{{d_2^1}'} \ar[d]^{q_1^1} && {F_1^1}' \ar[d]^{d^1_1} \\
\cdots \ar[r]^{m_k^1} &G_{k-1}^1 \ar[r]^{m_{k-1}^1} & \cdots \ar[r]^{m_2^1} & G_1^1 \ar[rr]^-{-m^1_1(-)\cdot d_1(e_0)} && R, \\}\end{equation}
where ${F_1^1}' = F_1' \oplus \Big( \bigoplus_{i=2}^t Re_0^i \Big)$ and ${d_2^1}' = d_2' + d_0^2 + \cdots + d_0^t$. Proceed by induction on $t$. Observe that Theorem \ref{theres} is the base case $t=1$. Let $t>1$ and recall the notation of Setup \ref{setup4}. We may write
$$d_2^1 = \begin{pmatrix} d_2' & 0 \\
-q_1^1 & m_2^1 \\
\end{pmatrix} + \begin{pmatrix} d_0^2 & 0 \\
0 & 0 \\
\end{pmatrix} + \cdots + \begin{pmatrix} d_0^t & 0 \\
0 & 0 \\
\end{pmatrix}$$
where for each $i=2,\dots , t$,
$$\begin{pmatrix} d_0^i & 0 \\
0 & 0 \\
\end{pmatrix} : F_2^1 \to Re_0^i.$$
This means we are in the situation of Setup \ref{setup4}, only instead trimming $t-1$ generators from the ideal $K'+\mfa_1 K_0^1 + K_0^2 + \cdots + K_0^t$. Observe that the maps $\begin{pmatrix} q^{i}_j & 0 \end{pmatrix} : F_{j+1}^1 = F_{j+1} \oplus G_{j+1}^1 \to G_j^i$ make the diagram of Proposition \ref{ittheyexist} commute. By induction, the mapping cone of
$$\xymatrix{\cdots \ar[r]^{d^1_{k+1}} &  F^1_{k} \ar[ddd]^{\begin{pmatrix} q_{k-1}^2 & 0 \\
\vdots \\
q_{k-1}^t & 0 \\
\end{pmatrix}}\ar[r]^{d^1_{k}} & \cdots \ar[r]^{d^1_3} & F^1_2 \ar[rrrr]^{\begin{pmatrix} d_2' & 0 \\
-q_1^1 & m_2^1 \\
\end{pmatrix}} \ar[ddd]^{\begin{pmatrix} q_1^2 & 0 \\
\vdots \\
q_1^t & 0 \\
\end{pmatrix}} &&&& F_1' \ar[ddd]^{d_1-m^1_1(-)\cdot d_1(e^1_0)} \\
&&&&&&& \\
&&&&&&& \\
\cdots \ar[r]^-{\bigoplus m^i_k} & \bigoplus_{i=2}^t G^i_{k-1} \ar[r]^-{\bigoplus m^i_{k-1}} & \cdots \ar[r]^-{\bigoplus m^i_2} & \bigoplus_{i=2}^t G^i_1 \ar[rrrr]^-{-\sum_{i=2}^t m^i_1(-)\cdot d_1(e^i_0)} &&&& R \\}$$
forms a resolution of $K'+ \mfa_1 K_0^1 + \big( \mfa_2 K_0^2 + \cdots + \mfa_t K_0^t \big)$ (recall that the top row forms a resolution of $K'+ (K_0^2 + \cdots + K_0^t) + \mfa_1 K_0^1$ by Theorem \ref{theres}). The differentials of this mapping cone are the same as the differentials induced by the mapping cone of diagram \ref{itcomx} as in the statement of Theorem \ref{itres}.
\end{proof}

\begin{definition}\label{def:ittrimcx}
The \emph{iterated trimming complex} associated to the data of Setup \ref{setup4} is the complex of Theorem \ref{itres}.
\end{definition}

As an immediate consequence, one obtains the following result (the proof of which is identical to that of Corollary \ref{torrk}):

\begin{cor}\label{ittorrk}
Adopt notation and hypotheses of Setup \ref{setup4}. Assume furthermore that the complexes $F_\bullet$ and $G_\bullet$ are minimal. Then for $i \geq 2$,
$$\dim_k \tor_i^R (R/J , k) = \rank F_i + \sum_{j=1}^t  \rank G^j_i - \rank \Bigg( \begin{pmatrix} q_i^1 \\
\vdots \\
q_i^t \\
\end{pmatrix} \otimes k\Bigg) - \rank \Bigg( \begin{pmatrix} q_{i-1}^1 \\
\vdots \\
q_{i-1}^t \\
\end{pmatrix} \otimes k\Bigg),$$
and
$$\mu (J) = \mu(I) -t+ \sum_{j=1}^t \mu(\mfa_j) -\rank \Bigg( \begin{pmatrix} q_1^1 \\
\vdots \\
q_1^t \\
\end{pmatrix} \otimes k\Bigg) .\qquad \qquad \square $$ 
\end{cor}

\begin{example}
Let $R = k[x,y,z]$, $$X =\begin{pmatrix}
      0&0&0&{-x^{2}}&{-z^{2}}\\
      0&0&{-x^{2}}&{-z^{2}}&{-y^{2}}\\
      0&x^{2}&0&{-y^{2}}&0\\
      x^{2}&z^{2}&y^{2}&0&0\\
      z^{2}&y^{2}&0&0&0\end{pmatrix},$$
and $I = \textrm{Pf} (X)$, the ideal of submaximal pfaffians of $X$. Let $F_\bullet$ denote the complex
$$\xymatrix{0 \ar[r] & R \ar[r]^-{d_1^*} & R^n \ar[r]^-{X} & R^n \ar[r]^{d_1} & R,}$$
with 
$$d_1 = \begin{pmatrix}
      y^{4}&{-y^{2}z^{2}}&-x^{2}y^{2}+z^{4}&{-x^{2}z^{2}}&x^{4}\end{pmatrix}.$$
      This is a minimal free resolution of $R/I$ (see \cite{buchsbaum1977algebra}). In the notation of Setup \ref{setup4}, let
      $$K' := (-x^2y^2+z^4,-x^2z^2,x^4), \ K_0^1 := (y^4), \ K_0^2 := (-y^2z^2),$$
      and $\mfa_1 = \mfa_2 := (x,y,z)$. Let $G^1_\bullet = G_\bullet^2$ denote the Koszul complex:
      $$\xymatrix{0 \ar[r] & R \ar[r]^-{\begin{pmatrix}
      z\\
      {-y}\\
      x\end{pmatrix}} \ar[r] & R^3 \ar[rrr]^-{\begin{pmatrix}
      {-y}&{-z}&0\\
      x&0&{-z}\\
      0&x&y\end{pmatrix}} &&& R^3 \ar[rrr]^-{\begin{pmatrix}
      x&y&z\end{pmatrix}} &&& R}.$$
      Then, one computes:
      $$q^1_1 = \begin{pmatrix}
      0&0&0&{-x}&0\\
      0&0&0&0&0\\
      0&0&0&0&{-z}\end{pmatrix} : R^5 \to R^3,$$
      $$q^1_2 = \begin{pmatrix}
      0\\
      {-x^{3}z}\\
      0\end{pmatrix} : R \to R^3,$$
      $$q^2_1 = \begin{pmatrix}
      0&0&{-x}&0&0\\
      0&0&0&0&{-y}\\
      0&0&0&{-z}&0\end{pmatrix}: R^5 \to R^3,$$
      $$q^2_2 = \begin{pmatrix}
      {-x^{3}y}\\
      x\,z^{3}\\
      0\end{pmatrix}: R \to R^3.$$
      Then, the mapping cone of Theorem \ref{itres} forms a resolution of $R/(K' + \mfa_1K_0^1 + \mfa_2 K_0^2)$. In particular, we deduce that this mapping cone is a minimal free resolution and hence the above quotient ring has Betti table
      $$\begin{matrix}
      &0&1&2&3\\\text{total:}&1&9&11&3\\\text{0:}&1&\text{.}&\text{.}&\text{.}\\\text{1:}&\text{.}&\
      \text{.}&\text{.}&\text{.}\\\text{2:}&\text{.}&\text{.}&\text{.}&\text{.}\\\text{3:}&\text{.}&3
      &\text{.}&\text{.}\\\text{4:}&\text{.}&6&11&2\\\text{5:}&\text{.}&\text{.}&\text{.}&\text
      {.}\\\text{6:}&\text{.}&\text{.}&\text{.}&\text{.}\\\text{7:}&\text{.}&\text{.}&\text{.}&1.\\\
      \end{matrix}$$
\end{example}

\section{Betti Tables for Ideals Obtained by Removing a Generator from Generic Submaximal Pfaffian Ideals and Ideals of Maximal Minors}\label{ressubs}

In this section, we demonstrate how to use trimming complexes to compute the Betti table of the ideal generating by removing a single generator from a given generating set of an ideal $I$.

\begin{setup}\label{setup5}
Let $R=k[x_1,\dots, x_n]$ be a standard graded polynomial ring over a field $k$, with $R_+ := R_{>0}$. Let $I \subseteq R$ be a homogeneous $R_+$-primary ideal and $(F_\bullet, d_\bullet)$ denote a homogeneous free resolution of $R/I$. 

Write $F_1 = F_1' \oplus Re_0$, where $e_0$ generates a free direct summand of $F_1$. Using the isomorphism
$$\hom_R (F_2 , F_1 ) = \hom_R (F_2,F_1') \oplus \hom_R (F_2 , Re_0)$$
write $d_2 = d_2' + d_0$, where $d_2' \in \hom_R (F_2,F_1')$, $d_0 \in \hom_R (F_2 , Re_0)$. Let $\mfa$ denote a homogeneous ideal with
$$d_0 (F_2) = \mfa e_0,$$
and $(G_\bullet , m_\bullet)$ be a homogeneous free resolution of $R/\mfa$. 

Use the notation $K' := \im (d_1|_{F_1'} : F_1' \to R)$, $K_0 := \im (d_1|_{Re_0} : Re_0 \to R)$, and let $J := K' + \mfa \cdot K_0$.
\end{setup}

\begin{prop}
Adopt notation and hypotheses as in Setup \ref{setup5}. Then the resolution of Theorem \ref{theres} resolves $K'$.
\end{prop}

\begin{proof}
It will be shown that $\mfa = K' : K_0$. Observe that $K' : K_0 \subseteq \mfa$ by Proposition \ref{colon}. Let $r \in \mfa$; by assumption, there exists $f \in F_2$ such that $d_0 (f) = r e_0$. Since $F_\bullet$ is a complex, $d_1 ( re_0) = - d_1 (d_2' (f))$, so that $r K_0 \subseteq K'$. This yields that $\mfa = K' : K_0$. In particular, we find that $\mfa K_0 \subset K'$. The resolution of Theorem \ref{theres} resolves $K'+ \mfa K_0 = K'$, so the result follows.
\end{proof}

 \begin{notation}
 Given a skew symmetric matrix $X \in M_n (R)$, where $R$ is some commutative ring, the notation $\textrm{Pf}_j (X)$ will denote the pfaffian of the matrix obtained by removing the $j$th row and column from $X$. 
 \end{notation}

\begin{prop}\label{prop:btabforpfaffs}
Let $R=k[x_{ij} \mid 1\leq i<j \leq n]$ and let $X$ denote a generic $n\times n$ skew symmetric matrix, with $n \geq 7$ odd. Given $1 \leq i \leq n$, the ideal 
$$J:= (\textrm{Pf}_j (X) \mid i \neq j )$$
has Betti table
\begin{equation*}\begin{tabular}{C|C C C C C C C C C}
     & 0 & 1 & 2 & 3 &\cdots & k & \cdots & n-1  \\
     \hline 
   0  & 1 & \cdot & \cdot &  &\cdots & \cdot & \cdots  & \cdot \\
   
   \vdots &  &  & & & & & & \\
   
   (n-3)/2 & \cdot & n-1 & 1 & \cdot & \cdots &  & \cdots  & \cdot \\
   
   \vdots & & & & & & & & \\
   
   (n-1)/2 & \cdot & \cdot  & \binom{n-1}{2}& \binom{n-1}{3} &  \cdots & \binom{n-1}{k} & \cdots  & 1 \\
   
   \vdots & & & & & & & & \\
   
   n-3 & \cdot & \cdot & \cdot & 1 & \cdots & \cdot & \cdots & \cdot  \\
\end{tabular}
\end{equation*}
In the case where $n=5$, the Betti table is
\begin{equation*}
    \begin{tabular}{L|L L L L L}
           & 0 & 1 &2&3&4\\
           \hline
          0 & 1 &\cdot&\cdot&\cdot&\cdot \\
          1 & \cdot & 4 & 1 & \cdot &\cdot \\
          2 & \cdot & \cdot & 6 & 5 & 1 \\
    \end{tabular}
\end{equation*}
\end{prop}

\begin{proof}
In view of Corollary \ref{torrk}, it suffices to compute the ranks of the maps $q_i \otimes k$ for all appropriate $i$. Let $F_\bullet$ denote the minimal free resolution of the ideal of submaximal pfaffians of $X$. Observe that $F_\bullet$ is of the form
$$\xymatrix{0 \ar[r] & R \ar[r]^-{d_1^*} & R^n \ar[r]^-{X} & R^n \ar[r]^{d_1} & R,}$$
where $d_1= (\textrm{Pf}_1 (X), - \textrm{Pf}_2 (X) , \dotsc , (-1)^{n+1} \textrm{Pf}_n (X))$ (see, for instance, \cite{buchsbaum1977algebra}). Fix an integer $1 \leq \ell \leq n$ and let $K' := (\textrm{Pf}_i (X) \mid i \neq \ell )$, $K_0 := (\textrm{Pf}_\ell (X))$. Observe that $\ell$th row of $X$ generates the ideal
$$\begin{cases}
(x_{12} , \dotsc , x_{1n}) \quad \textrm{if} \ \ell=1, \\
(x_{1\ell}, \dotsc , x_{\ell-1,\ell} , x_{\ell,\ell+1}, \dotsc , x_{\ell,n}) \quad \textrm{if} \ 1<\ell<n, \\
(x_{1n} , \dotsc , x_{n-1,n} ) \quad \textrm{if} \ \ell=n. \\
\end{cases}$$
Notice that this ideal is a complete intersection on $n-1$ generators; in the notation of Setup \ref{setup5}, the ideal $\mfa$ is this complete intersection (so that $\mfa K_0 \subseteq K'$). Let $G_\bullet$ denote the Koszul complex resolving $\mfa$.

Observe that for $i \geq 3$,
$$q_i : F_{i+1} = 0 \to G_i,$$
so $q_i \otimes k = 0$ for $i \geq 3$. By counting degrees, one finds $q_2 \otimes k = 0$. Finally, the map $q_1$ is simply the projection
$$q_1 : F_2 \cong R^{n} \to G_1 \cong R^{n-1}$$
onto the appropriate summands; this map has $\rank (q_1 \otimes k) = n-1$. Combining this information with Corollary \ref{torrk} and Remark \ref{rk:gradedBnos}, for $i \geq 4$,
$$\dim_k \tor_i^R (R/J , k) = \binom{n-1}{i}.$$
For $i=3$ and $n \geq 7$,
\begingroup\allowdisplaybreaks
\begin{align*}
    \dim_k \tor_3^R (R/J)_{(n+5)/2} &= \rank G_3 \\
    &= \binom{n-1}{3} \\
     \dim_k \tor_3^R (R/J)_{n} &= \rank F_3 \\
     &= 1 .\\
\end{align*}
\endgroup
For $i=3$ and $n=5$, observe that $n= (n+5)/2$, so
\begingroup\allowdisplaybreaks
\begin{align*}
    \dim_k \tor_3^R (R/J) &= \rank F_3 + \rank G_3 \\
    &=1 +  \binom{n-1}{3} = 5. \\
\end{align*}
\endgroup
Finally, for $i=2$ and $n \geq 5$,
\begingroup\allowdisplaybreaks
\begin{align*}
    \dim_k \tor_2^R (R/J)_{(n+1)/2} &= \rank F_2 - \rank (q_1 \otimes k) \\
    &= n- (n-1) = 1 \\
     \dim_k \tor_2^R (R/J)_{(n+3)/2} &= \rank G_2 \\
     &= \binom{n-1}{2} .\\
\end{align*}
\endgroup
\end{proof}

Observe the difference between the Betti table of Proposition \ref{prop:btabforpfaffs} and the classical case of the ideal generated by all submaximal pfaffians of a generic skew symmetric matrix. In the latter case, this ideal is always a grade $3$ Gorenstein ideal (in particular, the projective dimension is $3$). After removing a generator, one sees that the projective dimension can become arbitrarily large based on the size of the matrix $X$.

Next, we want to compute the graded Betti numbers when removing a generator from an ideal of maximal minors of a generic $n \times m$ matrix. This case requires more work since the $q_\ell$ maps of Proposition \ref{theyexist} must be computed explicitly in order to compute the ranks. For convenience, we recall the definition of the Eagon-Northcott complex.

\begin{notation}
Let $V$ be a $k$-vector space, where $k$ is any field. The notation $\bigwedge^i V$ denotes the $i$th exterior power of $V$ and $D_i (V)$ denotes the $i$th divided power of $V$ (see \cite[Section A2.4]{eisenbud2013commutative} for the definition of $D_i(V)$).
\end{notation}

\begin{definition}\label{EN}
Let $\phi : F \to G$ be a homomorphism of free modules of ranks $f$ and $g$, respectively, with $f \geq g$. Let $c_\phi$ be the image of $\phi$ under the isomorphism $\Hom_R (F,G) \xrightarrow{\cong} F^* \otimes G $. The \emph{Eagon-Northcott complex} is the complex
$$0 \to D_{f-g} (G^*) \otimes \bigwedge^f F \to D_{f-g-1} (G^*) \otimes \bigwedge^{f-1} F \to \cdots \to G^* \otimes \bigwedge^{g+1} F \to \bigwedge^g F \to \bigwedge^g G$$
with differentials in homological degree $\geq 2$ induced by multiplication by the element $c_\phi \in F^* \otimes G$, and the map $\bigwedge^g F \to \bigwedge^g G$ is $\bigwedge^g \phi$.
\end{definition}

\begin{setup}\label{setup6}
Let $R= k[x_{ij} \mid 1\leq i \leq n, 1 \leq j \leq m]$ and $M = (x_{ij})_{1\leq i \leq n, 1 \leq j \leq m}$ denote a generic $n \times m$ matrix, where $n \leq m$. View $M$ as a homomorphism $M : F \to G$ of free modules $F$ and $G$ of rank $m$ and $n$, respectively.

Let $f_i$, $i=1, \dots , m$, $g_j$, $j=1 , \dots , n$ denote the standard bases with respect to which $M$ has the above matrix representation. Write 
$$\bigwedge^n F = F' \oplus Rf_\sigma$$
for some free module $F'$, where $\sigma = (\sigma_1 < \cdots < \sigma_n)$ is a fixed index set, and the notation $f_\sigma$ denotes $f_{\sigma_1} \w \cdots \w f_{\sigma_n}$. Recall that the Eagon-Northcott complex of Definition \ref{EN} resolves the quotient ring defined by $I_n (M)$, the ideal of $n \times n$ minors of $M$. 

We will consider the submodule of $\bigwedge^{n+\ell} F$ generated by all elements of the form $f_{\sigma,\tau}$, where $\tau= (\tau_1 < \cdots < \tau_\ell)$ and $\sigma \cap \tau = \varnothing$. The notation $f_{\sigma,\tau}$ denotes the element $f_\sigma \w f_\tau$. If $\tau = (\tau_1 < \cdots < \tau_n)$, let $\Delta_\tau$ denote the determinant of the matrix formed by columns $\tau_1 , \dots , \tau_n$ of $M$. Then, in the notation of Setup \ref{setup5},
$$K' = (\Delta_\tau \mid \tau \neq \sigma)$$
and $K_0 = (\Delta_\sigma)$.

Observe that the Eagon-Northcott differential $d_2 : G^* \otimes \bigwedge^{n+1} F \to \bigwedge^n F$ induces a homomorphism $d_0 : G^* \otimes \bigwedge^{n+1} F \to Rf_\sigma$ by sending
$$g_i^* \otimes f_{\{ j \} ,\sigma   } \mapsto x_{ij} f_\sigma,$$
and all other basis elements to $0$. In the notation of Setup \ref{setup5}, 
$$\mfa = ( x_{ij} \mid i=1, \dots , n, \ j \notin \sigma ).$$
This means $\mfa$ is a complete intersection generated by $n(m-n)$ elements, and hence is resolved by the Koszul complex. Moreover, $\mfa K_0 \subseteq K'$. Let
$$U = \bigoplus_{\substack{1 \leq i \leq n \\ j \notin \sigma}} Re_{ij}$$
with differential induced by the homomorphism $m_1 : U \to R$ sending $e_{ij} \mapsto x_{ij}$. If $L=(i,j)$ is a $2$-tuple, then the notation $e_L$ will denote $e_{ij}$. 
\end{setup}

The proof of the following Proposition is a straightforward computation.
\begin{prop}
Adopt notation and hypotheses of Setup \ref{setup6}. Define $q_1 : G^* \otimes \bigwedge^{n+1} F \to U$ by sending $g_i^* \otimes f_{ \{ j \} ,\sigma  } \mapsto e_{ij}$ and all other basis elements to $0$. Then the following diagram commutes:
$$\xymatrix{& G^* \otimes \bigwedge^{n+1} F \ar[dl]_-{q_1} \ar[d]^{d'_0} \\
U \ar[r]_{m_1} & \mfa, \\}$$
where $d_0' : G^* \otimes \bigwedge^{n+1} F \to R$ is the composition
$$\xymatrix{G^* \otimes \bigwedge^{n+1} F \ar[r]^-{d_0} & Rf_\sigma \ar[r] & R\\},$$
and where the second map sends $f_\sigma \mapsto 1$.
\end{prop}

We will need the following definition before introducing the $q_i$ maps for $i \geq 2$.
\begin{definition}\label{lsubset}
Let $\tau = (\tau_1 , \dots , \tau_\ell)$ be an indexing set of length $\ell$ with $\tau_1 < \cdots < \tau_\ell$. Let $\alpha = (\alpha_1 , \cdots , \alpha_n)$, with $\alpha_i \geq 0$ for each $i$. Define $\cat{L}_{\alpha,\tau}$ to be the subset of size $\ell$ subsets of the cartesian product
$$\{ i \mid \alpha_i \neq 0 \} \times \tau,$$
where $\{ (r_1,\tau_1) , \dots , (r_\ell, \tau_\ell ) \} \in \cat{L}_{\alpha,\tau}$ if $| \{i \mid r_i = j \} | = \alpha_j$. 

Observe that $L_{\alpha , \tau}$ is empty unless $\alpha_1 + \cdots + \alpha_n = \ell$. 
\end{definition}

\begin{example}
One easily computes:
$$\cat{L}_{(2,0,1),(1,2,3)} = \{ \{(3, 1), (1, 2), (1, 3)\}, \{(1, 1),  (3, 2),(1, 3) \}, \{ (1, 1), (1, 2), (3, 3)\} \}$$
\begingroup\allowdisplaybreaks
\begin{align*}
    \cat{L}_{(2,0,2),(1,2,3,4)} =& \{\left\{\left(3,1\right),\,\left(3,2\right),\left(1,3\right),\,\left(1,4\right)\right\},\,\left\{\left(3,1\right),\,\left(1,2\right),\,\left(3,3\right),\left(1,4\right)\right\}, \\
     &\left\{\left(1,1\right),\,\left(3,2\right),\,\left(3,3\right),\left(1,4\right)\right\},\,\left\{\left(3,1\right),\,\left(1,2\right),\,\left(1,3\right),\,\left(3,4\right)\right\}, \\
     &\left\{\left(1,1\right),\,\left(3,2\right),\,\left(1,3\right),\,\left(3,4\right)\right\},\,\left\{\left(1,1\right),\,\left(1,2\right),\,\left(3,3\right),\,\left(3,4\right)\right\}\}
\end{align*}
\endgroup
\end{example}

\begin{lemma}\label{reduct}
Let $\tau = (\tau_1 , \dots , \tau_\ell)$ be an indexing set of length $\ell$ with $\tau_1 < \cdots < \tau_\ell$. Let $\alpha = (\alpha_1 , \cdots , \alpha_n)$, with $\alpha_i \geq 0$ for each $i$. Use the notation $\alpha^i := (\alpha_1 , \dots , \alpha_i-1 , \dots , \alpha_n)$. Then any $L' \in \cat{L}_{\alpha^i , \tau \backslash \tau_k}$ is contained in a unique element $L \in \cat{L}_{\alpha , \tau}$. 
\end{lemma}

\begin{proof}
Given $L'$, take $L := L' \cup (i , j_k)$, ordered appropriately. Assume that $L' \subseteq L''$ for some other $L'' \in \cat{L}_{\alpha , \tau}$. It is easy to see that $L'' \backslash L' = (a,\sigma_k)$, where $a$ is some integer. However, since $\alpha^i$ differs by $\alpha$ by $1$ in the $i$th spot, $a=i$, whence $L = L''$ and $L$ is unique.
\end{proof}

\begin{lemma}
Adopt notation and hypotheses of Setup \ref{setup6}. Define
$$q_\ell : D_\ell (G^*) \otimes \bigwedge^{n+\ell} F \to \bigwedge^\ell U, \quad \ell \geq 2,$$
by sending 
$$g_1^{*(\alpha_1)} \cdots g_{n}^{*(\alpha_n)} \otimes f_{\tau,I\sigma} \mapsto  \sum_{L \in \cat{L}_{\alpha,\tau}} e_{L_1} \w \cdots \w e_{L_\ell},$$
where $\cat{L}_{\alpha,\tau}$ is defined in Definition \ref{lsubset}. All other basis elements are sent to $0$. Then the following diagram commutes:
$$\xymatrix{D_\ell (G^*) \otimes \bigwedge^{n+\ell} F \ar[d]_{q_\ell} \ar[r]^-{d_{\ell}} & D_{\ell-1} (G^*) \otimes \bigwedge^{n+\ell-1} F \ar[d]^{q_{\ell-1}} \\
\bigwedge^\ell U \ar[r]_{m_\ell} & \bigwedge^{\ell-1} U. \\}$$
\end{lemma}

\begin{proof}
We first compute the image of the element
$$g_1^{*(\alpha_1)} \cdots g_n^{*(\alpha_n)} \otimes f_{\tau,\sigma}$$
going clockwise about the diagram. We obtain:
\begingroup\allowdisplaybreaks
\begin{align*}
    g_1^{*(\alpha_1)} \cdots g_n^{*(\alpha_n)} \otimes f_{\tau,\sigma} &\mapsto \sum_{\substack{\{ i \mid \alpha_i \neq 0 \} \\ 1 \leq j \leq \ell }} (-1)^{j+1} x_{i\tau_j} g_1^{*(\alpha_1)} \cdots g_i^{*(\alpha_i-1)} \cdots g_n^{*(\alpha_n)} \otimes f_{\tau \backslash \tau_j,\sigma} \\
    &+\sum_{\substack{\{ i \mid \alpha_i \neq 0 \} \\ 1 \leq j \leq n }} (-1)^{m-n+j+1} x_{i \sigma_j} g_1^{*(\alpha_1)} \cdots g_i^{*(\alpha_i-1)} \cdots g_n^{*(\alpha_n)} \otimes f_{J ,\sigma\backslash \sigma_j} \\
    &\mapsto \sum_{\substack{\{ i \mid \alpha_i \neq 0 \} \\ 1 \leq j \leq \ell }} \sum_{L \in \cat{L}_{\alpha^i,\tau\backslash \tau_j} }  (-1)^{j+1} x_{i \tau_j} e_{L_1} \w \cdots \w e_{L_{\ell-1}} \\ 
\end{align*}
\endgroup
where in the above, denote $\alpha^i := (\alpha_1 , \dots , \alpha_i-1 , \dots , \alpha_n)$ and $L_i$ the $i$th entry of $L \in \cat{L}_{\alpha , \tau \backslash \tau_j}$. According to Lemma \ref{reduct},
\begingroup\allowdisplaybreaks
\begin{align*}
    &\sum_{\substack{\{ i \mid \alpha_i \neq 0 \} \\ 1 \leq j \leq \ell }} \sum_{L \in \cat{L}_{\alpha^i,\tau \backslash \tau_j} }  (-1)^{j+1} x_{i \tau_j} e_{L_1} \w \cdots \w e_{L_{\ell-1}} \\
    =& \sum_{ 1 \leq j \leq \ell } \sum_{L \in \cat{L}_{\alpha,\tau} }  (-1)^{j+1} x_{L_j} e_{L_1} \w \cdots \w \widehat{e_{L_j}} \w \cdots \w e_{L_{\ell}}. \\
\end{align*}
\endgroup
Moving in the counterclockwise direction, we obtain:
\begingroup\allowdisplaybreaks
\begin{align*}
    g_1^{*(\alpha_1)} \cdots g_n^{*(\alpha_n)} \otimes f_{\tau,\sigma} &\mapsto  \sum_{L \in \cat{L}_{\alpha,\tau}}  e_{L_1} \w \cdots \w e_{L_\ell} \\
    &\mapsto  \sum_{L \in \cat{L}_{\alpha,\tau}} \sum_{1 \leq j \leq \ell}  (-1)^{j+1} x_{L_j} e_{L_1} \w \cdots \w \widehat{e_{L_j}} \w \cdots \w e_{L_\ell} \\ 
\end{align*}
\endgroup
\end{proof}

\begin{lemma}\label{qranks}
Adopt notation and hypotheses of Setup \ref{setup6}. Then the maps
$$q_\ell : D_{\ell} (G^*) \otimes \bigwedge^{n+\ell} F \to \bigwedge^\ell U$$
have $\rank (q_\ell \otimes k) = \binom{n+\ell-1}{\ell} \cdot \binom{m-n}{\ell}$ for all $\ell =1, \cdots , m-n+1$ and $\rank (q_\ell \otimes k ) = 0$ for all $\ell = m-n+2, \dots , n(m-n)$. 
\end{lemma}

\begin{remark}
If we use the convention that $\binom{r}{s} = 0$ for $s>r$, then the above says that $\rank (q_\ell \otimes k) = \binom{n+\ell-1}{\ell} \cdot \binom{m-n}{\ell}$ for all $\ell \geq 1$.
\end{remark}

\begin{proof}
First observe that since the Eagon-Northcott complex is $0$ in homological degrees $\geq m-n+2$, it is immediate that $q_\ell = 0$ for $\ell \geq m-n+2$. For the first claim, this follows from the fact that for $\alpha \neq \alpha'$,
$$\cat{L}_{\alpha,\tau}\cap \cat{L}_{\alpha',\tau} = \varnothing,$$
which implies that the image of each element $g_1^{*(\alpha_1)} \cdots g_n^{*(\alpha_n)} \otimes f_{\tau,\sigma} \in D_{\ell} (G^*) \otimes \bigwedge^{n+\ell} F$ under $q_\ell$ has maximal rank (recall that these are the only elements with nonzero image). This rank is computed by counting all such basis elements; it is clear that there are $\binom{m-n}{\ell}$ possible elements of the form $f_{\tau,\sigma}$, since $\sigma$ is a fixed index set of length $n$. The rank of $D_\ell (G^*)$ is $\binom{n+\ell-1}{\ell}$, therefore we conclude that the rank of each $q_\ell$ is
$$\binom{n+\ell-1}{\ell} \binom{m-n}{\ell}$$
\end{proof}

\begin{theorem}\label{thm:btabforsinglemaxlminor}
Adopt notation and hypotheses of Setup \ref{setup6}. If $\tau = (\tau_1 < \cdots < \tau_n)$, let $\Delta_\tau$ denote the determinant of the matrix formed by columns $\tau_1 , \dots , \tau_n$ of $M$. Then the ideal
$$K' := (\Delta_\tau \mid \tau \neq \sigma)$$
has Betti table
\begin{equation*}\begin{tabular}{C|C C C C C C C C C}
     & 0 & 1 & \cdots &  \ell & \cdots & n(m-n)-1 & n(m-n)  \\
     \hline 
   0  & 1 & \cdot & \cdots & \cdot &\cdots & \cdot &  \cdot \\
 
 \vdots & & & & & & & \\
   
   n-1 & \cdot & \binom{m}{n}-1 & \cdots & \binom{n+\ell-2}{\ell-1}\Big( \binom{m}{n+\ell-1} -  \binom{m-n}{\ell-1}\Big) & \cdots & \cdot   & \cdot \\
   
   n & \cdot & \cdot  & \cdots & \binom{n(m-n)}{\ell} - \binom{n+\ell-1}{\ell} \binom{m-n}{\ell} &  \cdots & n(m-n)   & 1 \\
\end{tabular}
\end{equation*}
\end{theorem}

\begin{proof}
We employ Corollary \ref{torrk} and Remark \ref{rk:gradedBnos}. By selection, $\mfa K_0 \subseteq K'$. Let $E_\bullet$ denote the Eagon-Northcott complex as in Setup \ref{setup6}. For $\ell \geq 1$,
$$\rank E_\ell = \binom{n+\ell-2}{\ell-1} \binom{m}{n+\ell-1}.$$
Similarly, let $K_\bullet$ denote the Koszul complex resolving $\mfa$ as in Setup \ref{setup6}. Then
$$\rank K_\ell = \binom{n(m-n)}{\ell}.$$
Combining the information above with that of Lemma \ref{qranks}, Corollary \ref{torrk}, and Remark \ref{rk:gradedBnos}, we have:
\begingroup\allowdisplaybreaks
\begin{align*}
    \dim_k \tor_\ell^R (R/K' , k)_{n+\ell} &= \rank E_\ell - \rank_k (q_{\ell-1} \otimes k) \\
    &= \binom{n+\ell-2}{\ell-1} \binom{m}{n+\ell-1}- \binom{n+\ell-2}{\ell-1} \binom{m-n}{\ell-1}, \\
    \dim_k \tor_\ell^R (R/K' , k)_{n+\ell+1} &= \rank K_{\ell} - \rank_k (q_{\ell} \otimes k) \\
    &=\binom{n(m-n)}{\ell} - \binom{n+\ell-1}{\ell} \binom{m-n}{\ell} . \\
\end{align*}
\endgroup
This concludes the proof.
\end{proof}

\section{Betti Tables for a Class of Determinantal Facet Ideals}\label{maxlminors}

In this section we consider the case for removing multiple generators from the ideal of maximal minors of a generic $n \times m$ matrix $M$. Such ideals belong to the class of ideals called determinantal facet ideals, which were studied in \cite{ene2013determinantal} and \cite{herzog2017linear}. Graded Betti numbers for these ideals appearing in higher degrees have not been previously computed, even in simple cases. In Theorem \ref{subres}, the graded Betti numbers of an infinite class of determinantal facet ideals defining quotient rings of regularity $n+1$ are computed explicitly in all degrees. In \cite{herzog2017linear} the linear strand for such ideals is computed in terms of the $f$-vector of some associated simplicial complex. We use the linear strand of Theorem \ref{subres} to deduce the $f$-vector of the simplicial complex associated to an $n$-uniform clutter obtained by removing pairwise disjoint subsets from all $n$-subsets of $[m]$ (see Corollary \ref{cor:fvect}).

\begin{setup}\label{setup7}
Let $R= k[x_{ij} \mid 1\leq i \leq n, 1 \leq j \leq m]$ and $M = (x_{ij})_{1\leq i \leq n, 1 \leq j \leq m}$ denote a generic $n \times m$ matrix, where $n \leq m$. View $M$ as a homomorphism $M : F \to G$ of free modules $F$ and $G$ of rank $m$ and $n$, respectively.

Fix indexing sets $\sigma_j = (\sigma_{j1}< \cdots < \sigma_{jn})$ for $j=1,\dots , r$ pairwise disjoint; that is, $\sigma_i \cap \sigma_j = \varnothing$ for $i \neq j$ (this intersection is taken as sets).  

Let $f_i$, for $i=1, \dots , m$, and $g_j$, for $j=1 , \dots , n$ denote the standard bases with respect to which $M$ has the above matrix representation. Write 
$$\bigwedge^n F =  F' \oplus Rf_{\sigma_1} \oplus \cdots \oplus Rf_{\sigma_r}$$
for some free module $F'$, where the notation $f_{\sigma_j}$ denotes $f_{\sigma_{j1}} \w \cdots \w f_{\sigma_{jn}}$. Recall that the Eagon-Northcott complex of Definition \ref{EN} resolves the quotient ring defined by $I_n(M)$. If $\tau = (\tau_1 < \cdots < \tau_n)$, let $\Delta_\tau$ denote the determinant of the matrix formed by columns $\tau_1 , \dots , \tau_n$ of $M$. Then, in the iterated version of Setup \ref{setup5},
$$K' = (\Delta_\tau \mid \tau \neq \sigma_j, \ j=1, \dotsc , r)$$
and $K_0^j = (\Delta_{\sigma_j})$.  

Observe that the Eagon-Northcott differential $d_2 : G^* \otimes \bigwedge^{n+1} F \to \bigwedge^m F$ induces homomorphisms $d^\ell_0 : G^* \otimes \bigwedge^{n+1} F \to Rf_{\sigma_j}$ by sending
$$g_i^* \otimes f_{ \{ \ell \},\sigma_j  } \mapsto x_{i\ell} f_{\sigma_j},$$
and all other basis elements to $0$. In the notation of Setup \ref{setup5}, this means we are considering the family of ideals
$$\mfa_j = ( x_{i\ell} \mid i=1, \dots , n, \ \ell \notin \sigma_j ).$$
For each $j=1, \dots , r$, $\mfa_j$ is a complete intersection generated by $n(m-n)$ elements, and hence is resolved by the Koszul complex. Let
$$U_j = \bigoplus_{\substack{1 \leq i \leq n \\ \ell \notin \sigma_j}} Re_{i\ell}$$
with differential induced by the homomorphism $m^j_1 : U_j \to R$ sending $e_{i\ell} \mapsto x_{i\ell}$. If $L=(i,j)$ is a $2$-tuple, then the notation $e_L$ will denote $e_{ij}$. 
\end{setup}

\begin{remark}
The assumption $\sigma_i \cap \sigma_j = \varnothing$ for $i \neq j$ implies that, if we define $K' := (\Delta_\tau \mid \tau \neq \sigma_j, \ j=1 , \dots , r )$, each $\Delta_{\sigma_j}$ satisfies $\mfa_j \Delta_{\sigma_j} \subset K'$. This is significant since it allows us to use the resolution of Theorem \ref{itres}. If the indexing sets were not pairwise disjoint, then we would have to apply Theorem \ref{theres} iteratively and compute the minimal presenting matrix at each step explicitly.
\end{remark}

\begin{prop}
Adopt notation and hypotheses of Setup \ref{setup7}. Define $q^j_1 : G^* \otimes \bigwedge^{n+1} F \to U_j$ by sending $g_i^* \otimes f_{ \{ \ell \}, \sigma_j } \mapsto e_{i\ell}$ and all other basis elements to $0$. Then the following diagram commutes:
$$\xymatrix{& G^* \otimes \bigwedge^{n+1} F \ar[dl]_-{q^j_1} \ar[d]^{d'_0} \\
U_j \ar[r]_{m_1^j} & \mfa_j, \\}$$
where $d_0' : G^* \otimes \bigwedge^{n+1} F \to R$ is the composition
$$\xymatrix{G^* \otimes \bigwedge^{n+1} F \ar[r]^-{d_0} & Rf_{\sigma_j} \ar[r] & R\\},$$
and where the second map sends $f_{\sigma_j} \mapsto 1$. 
\end{prop}

\begin{prop}
Adopt notation and hypotheses of Setup \ref{setup7}. Define
$$q^j_\ell : D_\ell (G^*) \otimes \bigwedge^{n+\ell} F \to \bigwedge^\ell U_j, \quad \ell \geq 2,$$
by sending 
$$g_1^{*(\alpha_1)} \cdots g_{n}^{*(\alpha_n)} \otimes f_{\tau,\sigma_j} \mapsto (-1)^{n} \sum_{L \in \cat{L}_{\alpha,\tau}} e_{L_1} \w \cdots \w e_{L_\ell},$$
where $\cat{L}_{\alpha,\tau}$ is defined in Definition \ref{lsubset}. All other basis elements are sent to $0$. Then the following diagram commutes:
$$\xymatrix{D_\ell (G^*) \otimes \bigwedge^{n+\ell} F \ar[d]_{q^j_\ell} \ar[r]^-{d_{\ell}} & D_{\ell-1} (G^*) \otimes \bigwedge^{n+\ell-1} F \ar[d]^{q^j_{\ell-1}} \\
\bigwedge^\ell U_j \ar[r]_{m^j_\ell} & \bigwedge^{\ell-1} U_j \\},$$
where $m_\ell^j$ is the standard Koszul differential induced by the map $m_1^j$ as in Setup \ref{setup7}.
\end{prop}

\begin{lemma}\label{itqrank}
Adopt notation and hypotheses of Setup \ref{setup7}. Then
$$\rank_k \Bigg( \begin{pmatrix}
q_\ell^1 \\
\vdots \\
q_\ell^r \\
\end{pmatrix}  \otimes k \Bigg) = \binom{n+\ell-1}{\ell} \cdot \sum_{i=1}^r (-1)^{i+1} \binom{r}{i} \binom{m-in}{\ell - (i-1)n}$$ 
\end{lemma}

\begin{proof}
For convenience, use the notation
$$\rk_\ell := \rank_k \Bigg( \begin{pmatrix}
q_\ell^1 \\
\vdots \\
q_\ell^r \\
\end{pmatrix}  \otimes k \Bigg).$$
As already noted, $\cat{L}_{\alpha' , \tau} \cap \cat{L}_{\alpha,\tau} = \varnothing$ for $\alpha \neq \alpha'$, so $\rk_\ell \leq r \binom{n+\ell -1}{\ell} \binom{m-n}{\ell}$. We want to count all elements
$$g_1^{*(\alpha_1)} \cdots g_n^{*(\alpha_n)} \otimes f_\tau$$
such that there exists $1 \leq j \leq r$ with 
$$0 \neq q_\ell^j (g_1^{*(\alpha_1)} \cdots g_n^{*(\alpha_n)} \otimes f_\tau),$$
taking into account the fact that some elements will have nonzero image under multiple $q^j_\ell$. Thus, we count all elements
$$g_1^{*(\alpha_1)} \cdots g_n^{*(\alpha_n)} \otimes f_\tau$$
such that the image under at least $i$ distinct $q_\ell^j$ is nonzero, then apply the inclusion exclusion principle.

It is easy to see that this set is obtained by choosing all indexing sets $\tau$ with $|\tau|=n+\ell$ such that $\tau = \sigma_{j_1} \cup \cdots \cup \sigma_{j_i} \cup \tau'$ for some $i$ and $\tau'$ with $\tau' \cap \sigma_{j_s} = \varnothing$ for each $s = 1, \dots , i$. Fixing $i$, there are $\binom{r}{i}$ unique choices for the union $\sigma_{j_1} \cup \cdots \cup \sigma_{j_i}$. For the indexing set $\tau'$, there are $m- in$ total choices of indices after removing all elements of the union $\sigma_{j_1} \cup \cdots \cup \sigma_{j_i}$, and we are choosing $\ell+n - in = \ell- (i-1)n$ elements. Using the inclusion-exclusion principle, the total number of indexing sets $\tau$ as above is
$$\sum_{i=2}^r (-1)^{i} \binom{r}{i} \binom{m-in}{\ell - (i-1)n}.$$
Multiplying by $\rank D_\ell (G)$ and subtracting from $r \binom{n+\ell -1}{\ell} \binom{m-n}{\ell}$, we obtain the result. 
\end{proof}

\begin{theorem}\label{subres}
Adopt notation and hypotheses as in Setup \ref{setup7}. Define 
$$\rk_\ell := \binom{n+\ell-1}{\ell} \cdot \sum_{i=1}^r (-1)^{i+1} \binom{r}{i} \binom{m-in}{\ell - (i-1)n}.$$
If $\tau = (\tau_1 < \cdots < \tau_n)$, let $\Delta_\tau$ denote the determinant of the matrix formed by columns $\tau_1 , \dots , \tau_n$ of $M$. Then the ideal
$$K' := (\Delta_\tau \mid \tau \neq \sigma_j, \ j=1 , \dots , r )$$
has Betti table
\begin{equation*}\begin{tabular}{C|C C C C C C C C C}
     & 0 & 1 & \cdots &\ell & \cdots & n(m-n)-1 & n(m-n)  \\
     \hline 
   0  & 1 & \cdot & \cdots &  \cdot &\cdots & \cdot &  \cdot \\
 
 \vdots & & & & & & & \\
   
   n-1 & \cdot & \binom{m}{n}-r & \cdots & \binom{n+\ell-2}{\ell-1} \binom{m}{n+\ell-1} - \rk_{\ell-1} & \cdots & \cdot   & \cdot \\
   
   n & \cdot & \cdot  & \cdots & r \cdot \binom{n(m-n)}{\ell} - \rk_\ell &  \cdots & r \cdot n(m-n)   & r \\
\end{tabular}
\end{equation*}
\end{theorem}

\begin{proof}
Let $E_\bullet$ denote the Eagon-Northcott complex as in Setup \ref{setup7}. For $\ell \geq 1$,
$$\rank E_\ell = \binom{n+\ell-2}{\ell-1} \binom{m}{n+\ell-1}.$$
Similarly, let $K^j_\bullet$ denote the Koszul complex resolving $\mfa_j$ as in Setup \ref{setup7}. Then for each $j=1, \dots , r$,
$$\rank K^j_\ell = \binom{n(m-n)}{\ell}.$$
Combining the information above with that of Lemma \ref{itqrank}, Corollary \ref{ittorrk}, and the iterated version of Remark \ref{rk:gradedBnos}, we have:
\begingroup\allowdisplaybreaks
\begin{align*}
    \dim_k \tor_\ell^R (R/K' , k)_{n+\ell} &= \rank E_\ell - \rank_k \Bigg( \begin{pmatrix}
q_{\ell-1}^1 \\
\vdots \\
q_{\ell-1}^r \\
\end{pmatrix}  \otimes k \Bigg) \\
    &= \binom{n+\ell-2}{\ell-1} \binom{m}{n+\ell-1}- \rk_{\ell-1}, \\
    \dim_k \tor_\ell^R (R/K' , k)_{n+\ell+1} &= \sum_{j=1}^r \rank K^j_{\ell} - \rank_k \Bigg( \begin{pmatrix}
q_\ell^1 \\
\vdots \\
q_\ell^r \\
\end{pmatrix}  \otimes k \Bigg) \\
    &=r \cdot \binom{n(m-n)}{\ell} - \rk_\ell . \\
\end{align*}
\endgroup
\end{proof}

The following definitions assume familiarity of the reader with the language of simplicial complexes. For an introduction, see, for instance, Chapter $5$ of \cite{bruns1998cohen}. Given a pure $(n-1)$-dimensional simplicial complex $\Delta$, the determinantal facet ideal $J_\Delta$ associated to $\Delta$ is generated by all maximal minors $\det (M_\tau)$, where $\tau = (\tau_1 < \cdots < \tau_n) \in \Delta$ is a facet of $\Delta$.

\begin{definition}\label{fvec}
Let $\Delta$ be a simplicial complex. The $f$-vector $(f_0 (\Delta) , \dots , f_{\dim \Delta} (\Delta) )$ is the sequence of integers with
$$f_i (\Delta ) = | \{ \sigma \in \Delta \mid \dim \sigma = i \} |.$$
\end{definition}

\begin{definition}
A \emph{clutter} $C$ on the vertex set $[n] := \{ 1 , \dots , n \}$ is a collection of subsets of $[n]$ such that no element of $C$ is contained in another. Any element of $C$ is called a \emph{circuit}. If all circuits of $C$ have the same cardinality $m$, then $C$ is called an $m$-\emph{uniform} clutter.

If $C$ is an $m$-uniform clutter, then a \emph{clique} of $C$ is a subset $\sigma$ of $[n]$ such that each $m$-subset $\tau$ of $\sigma$ is a circuit of $C$. 
\end{definition}

\begin{definition}
Let $M$ be a generic $n \times m$ matrix, with $n \leq m$. Given an $n$-uniform clutter $C$ on the vertex set $[m]$, associate to each circuit $\tau = \{ j_1 , \dots , j_n \}$ with $j_1 < \cdots < j_n$ the determinant $\det (M_\tau)$ of the submatrix formed by columns $j_1 , \dots , j_n$ of $M$. 

The ideal $J_C := \{ \det (M_\tau) \mid \tau \in C \}$ is called the \emph{determinantal facet ideal} associated to $C$. 

Similarly, define the \emph{clique complex} $\Delta (C)$ as the associated simplicial complex whose facets are the circuits of $C$. 
\end{definition}

The following definition is introduced in \cite{herzog2017linear}.

\begin{definition}\label{genEN}
Let $\phi : F \to G$ be a homomorphism of free modules of rank $m$ and $n$, respectively. Let $f_1, \dots , f_m$ and $g_1 , \dots , g_n$ denote bases of $F$ and $G$, respectively. Let $\Delta$ be a simplicial complex on the vertex set $[m]$. Then the \emph{generalized Eagon-Northcott complex} $\cat{C}_\bullet (\Delta ; \phi)$ associated to $\Delta$ is the subcomplex
$$0 \to \cat{C}_{m-n+1} \to \cdots \to \cat{C}_1 \to \cat{C}_0$$
of the Eagon-Northcott complex with $\cat{C}_0 = \bigwedge^n G$ and $\cat{C}_\ell \subseteq D_{\ell-1} (G^*) \otimes \bigwedge^{n+\ell-1} F$ for $\ell \geq 1$ the submodule generated by all elements $g_1^{*(\alpha_1)} \cdots g_{n}^{*(\alpha_n)} \otimes f_\sigma$, where $\sigma \in \Delta$ and $\dim \sigma = n+\ell-2$. 
\end{definition}

\begin{remark}\label{rkobs}
Notice that by definition of the $f$-vector in Definition \ref{fvec} combined with Definition \ref{genEN}, for $\ell \geq 1$,
$$\rank C_\ell ( \Delta ; \phi) = \binom{n+\ell-2}{\ell-1} f_{n+\ell-2} (\Delta).$$
\end{remark}

\begin{definition}
Let $F_\bullet$ be a minimal graded complex of free $R$-modules. The \emph{linear strand} $F_\bullet^\textrm{lin}$ of $F_\bullet$ is the complex with $F_i^\textrm{lin} =$ degree $i$ part of $F_i$, and differentials induced by the differentials of $F_\bullet$.
\end{definition}

The following result illustrates the connection between the complex of Definition \ref{genEN} and resolutions of determinantal facet ideals.

\begin{theorem}\cite[Theorem 4.1]{herzog2017linear}\label{linestrand}
Let $\phi : F \to G$ be a homomorphism of free modules of rank $m$ and $n$, respectively. Let $C$ be an $n$-uniform clutter on the vertex set $[m]$ with associated simplicial complex $\Delta(C)$. Let $J_C$ denote determinantal facet ideal associated to $C$, with minimal free resolution $\cat{F}_\bullet$. Then,
$$\cat{F}^{\textrm{lin}}_i = \cat{C}_i (\Delta(C) ; \phi),$$
where $\cat{F}^{\textrm{lin}}_\bullet$ denotes the linear strand of the complex $\cat{F}_\bullet$.
\end{theorem}

\begin{cor}\label{cor:fvect}
Let $C$ denote the $n$-uniform clutter on the vertex set $[m]$ obtained by removing $r$ pairwise disjoint elements from all $n$-subsets of $[m]$. Then for $\ell \geq 1$,
$$f_{n+\ell-2} (\Delta (C)) = \binom{m}{n+\ell-1} - \sum_{i=1}^r (-1)^{i+1} \binom{r}{i} \binom{m-in}{\ell - (i-1)n}$$
\end{cor}

\begin{proof}
Let $\phi : F \to G$ be a generic homomorphism of free modules of rank $m$ and $n$, respectively, and let $J_C$ denote the determinantal facet ideal associated to $C$ with minimal free resolution $\cat{F}_\bullet$. By Theorem \ref{subres} with $\ell \geq 1$,
$$\rank \cat{F}^{\textrm{lin}}_\ell =  \binom{n+ \ell-2}{\ell -1} \Bigg( \binom{m}{n+\ell-1} - \sum_{i=1}^r (-1)^{i+1} \binom{r}{i} \binom{m-in}{\ell - (i-1)n} \Bigg).$$
Combining this with Theorem \ref{linestrand} and Remark \ref{rkobs}, the result follows.
\end{proof}

\section*{Acknowledgements}

The author wishes to thank the anonymous referee for many helpful comments that greatly enhanced the final draft of this paper.

\bibliographystyle{amsplain}
\bibliography{biblio}

\providecommand{\bysame}{\leavevmode\hbox to3em{\hrulefill}\thinspace}
\providecommand{\MR}{\relax\ifhmode\unskip\space\fi MR }
\providecommand{\MRhref}[2]{%
  \href{http://www.ams.org/mathscinet-getitem?mr=#1}{#2}
}
\providecommand{\href}[2]{#2}
\begin{thebibliography}{10}

\bibitem{avramov2012cohomological}
Luchezar~L Avramov, \emph{A cohomological study of local rings of embedding
  codepth 3}, Journal of Pure and Applied Algebra \textbf{216} (2012), no.~11,
  2489--2506.

\bibitem{bruns1998cohen}
Winfried Bruns and J{\"u}rgen Herzog, \emph{Cohen-macaulay rings}, no.~39,
  Cambridge university press, 1998.

\bibitem{bruns2006determinantal}
Winfried Bruns and Udo Vetter, \emph{Determinantal rings}, vol. 1327, Springer,
  2006.

\bibitem{buchsbaum1977algebra}
David Buchsbaum and David Eisenbud, \emph{Algebra structures for finite free
  resolutions, and some structure theorems for ideals of codimension 3},
  American Journal of Mathematics \textbf{99} (1977), no.~3, 447--485.

\bibitem{christensen2019trimming}
Lars~Winther Christensen, Oana Veliche, and Jerzy Weyman, \emph{Trimming a
  gorenstein ideal}, Journal of Commutative Algebra \textbf{11} (2019), no.~3,
  325--339.

\bibitem{eisenbud2013commutative}
David Eisenbud, \emph{Commutative algebra: with a view toward algebraic
  geometry}, vol. 150, Springer Science \& Business Media, 2013.

\bibitem{ene2013determinantal}
Viviana Ene, J{\"u}rgen Herzog, Takayuki Hibi, and Fatemeh Mohammadi,
  \emph{Determinantal facet ideals}, The Michigan Mathematical Journal
  \textbf{62} (2013), no.~1, 39--57.

\bibitem{herzog2012ideals}
Juergen Herzog and Takayuki Hibi, \emph{Ideals generated by adjacent 2-minors},
  Journal of Commutative Algebra \textbf{4} (2012), no.~4, 525--549.

\bibitem{herzog2010binomial}
J{\"u}rgen Herzog, Takayuki Hibi, Freyja Hreinsd{\'o}ttir, Thomas Kahle, and
  Johannes Rauh, \emph{Binomial edge ideals and conditional independence
  statements}, Advances in Applied Mathematics \textbf{45} (2010), no.~3,
  317--333.

\bibitem{herzog2017linear}
J{\"u}rgen Herzog, Dariush Kiani, and Sara~Saeedi Madani, \emph{The linear
  strand of determinantal facet ideals}, The Michigan Mathematical Journal
  \textbf{66} (2017), no.~1, 107--123.

\bibitem{hocsten2004ideals}
Serkan Ho{\c{s}}ten and Seth Sullivant, \emph{Ideals of adjacent minors},
  Journal of Algebra \textbf{277} (2004), no.~2, 615--642.

\bibitem{vandebogert2019structure}
Keller VandeBogert, \emph{Structure theory for a class of grade 3 homogeneous
  ideals defining type 2 compressed rings}, Journal of Commutative Algebra (to
  appear).

\end{thebibliography}
\addcontentsline{toc}{section}{Bibliography}

\end{document}